\documentclass[10pt]{article}
\usepackage{amsmath,amscd,amsthm,amssymb}

\newtheorem{theorem}{Theorem}{}
\newtheorem{corollary}{Corollary}{}
{}
\newtheorem{remark}{Remark}{}
{}
\newtheorem{example}{Example}{}

\chardef\No=24

\begin{document}

\begin{center}
{\Large A new family of analytic functions defined by means of Rodrigues type formula}
\end{center}

\

{\centerline { Rabia Aktas, Abdullah Altin and Fatma Tasdelen}}

\

\centerline{\it Department of Mathematics, Ankara University, Ankara, Turkey}

\

{\centerline {E-mail addresses: raktas@science.ankara.edu.tr}}

\

{\centerline {altin@science.ankara.edu.tr }}

\

{\centerline {tasdelen@science.ankara.edu.tr }}

\

\begin{abstract} In this article, a class of analytic functions is investigated and their some
properties are established. Several recurrence relations and various classes
of bilinear and bilateral generating functions for these analytic functions
are also derived. Examples of some members belonging to this family of
analytic functions are given and differential equations satisfied by these
functions are also obtained.

\


\end{abstract}

\section{Introduction}

During the recent years, generalized and multivariable forms of the special
functions have important role in many branches of mathematics and mathematical
physics. Especially, special functions of mathematical physics and their
generalizations are often seen in physical problems.

For instance, Hermite polynomials described by Rodrigues formula below
\cite{R}
\begin{align}
H_{n}\, \left(  x\right)   &  =\left(  -1\right)  ^{n}e^{x^{2}}\frac{d^{n}%
}{dx^{n}}\left(  e^{-x^{2}}\right) \label{a1}\\
(n  &  =0,1,2,...)\nonumber
\end{align}
are generated by
\[
\sum \limits_{n=0}^{\infty}H_{n}\, \left(  x\right)  \frac{t^{n}}{n!}%
=\exp \left(  2xt-t^{2}\right)  .
\]
There are several applications of them in mathematics and physics. In
mathematics, they are met in probability, such as the Edgeworth series; in
combinatorics as an example of an Appell sequence and in the umbral calculus;
in physics, they arise in solution of the Schr\"{o}dinger equation for the
harmonic oscillator.

The most remarkable property of the Hermite polynomials is the fact that they
are orthogonal polynomials over the interval $\left(  -\infty,\infty \right)  $
with respect to the weight function $e^{-x^{2}}~$\cite{R}$.$ That is,%
\[%
%TCIMACRO{\dint \limits_{-\infty}^{\infty}}%
%BeginExpansion
{\displaystyle \int \limits_{-\infty}^{\infty}}
%EndExpansion
e^{-x^{2}}H_{n}\, \left(  x\right)  H_{m}\, \left(  x\right)  dx=0~~~,~~~n\neq
m,
\]
which is important for their use in quantum mechanics.

In \cite{AA1}, authors, being inspired by the Rodrigues formula for Hermite
polynomials, defined the family of polynomials $\phi_{k+n(m-1)}\, \left(
x\right)  $ via Rodrigues type formula%
\begin{align}
\phi_{k+n(m-1)}\, \left(  x\right)   &  =e^{\varphi_{m}\left(  x\right)
}\frac{d^{n}}{dx^{n}}\left(  \psi_{k}\left(  x\right)  e^{-\varphi_{m}\left(
x\right)  }\right) \label{aa2}\\
(n  &  =0,1,2,...)\nonumber
\end{align}
generated by%
\begin{equation}
\sum \limits_{n=0}^{\infty}\frac{\phi_{k+n(m-1)}\, \left(  x\right)  }{n!}%
t^{n}=\psi_{k}\left(  x+t\right)  e^{\varphi_{m}\left(  x\right)  -\varphi
_{m}\left(  x+t\right)  } \label{aa3}%
\end{equation}
where $\phi_{k+n(m-1)}\, \left(  x\right)  $ is a polynomial of degree
$k+n(m-1)$ since $\varphi_{m}\left(  x\right)  $ and $\psi_{k}\left(
x\right)  $ are polynomials of degree $m$ and $k$ , respectively.

Similar to $\left(  \ref{aa2}\right)  $, a class of polynomials with two
variables defined by Rodrigues type formula\ and their some properties were
investigated in \cite{AA2}. Moreover, a class of multivariable polynomials was
studied in \cite{GAA}.

Now, let's consider a family of analytic functions which is more general than
the polynomials $\left(  \ref{aa2}\right)  .$ Assume that $\varphi_{1}\left(
x\right)  ,~\varphi_{2}\left(  x\right)  $ and $\psi \left(  x\right)  $ are
analytic functions. Let a family of analytic functions be defined by Rodrigues
type formula
\begin{align}
\Theta_{n}\, \left(  x\right)   &  =\alpha^{\varphi_{1}\left(  x\right)
}\frac{d^{n}}{dx^{n}}\left(  \psi \left(  x\right)  \beta^{-\varphi_{2}\left(
x\right)  }\right) \label{aa4}\\
(n  &  =0,1,2,...)\nonumber
\end{align}
where $\alpha,\beta \in \mathbb{R}^{+}\backslash \left \{  1\right \}  $.

It is clear that the special case $\alpha=\beta=e,~\psi(x)=1,~\varphi
_{1}\left(  x\right)  =\varphi_{2}\left(  x\right)  =x^{2}$ gives the
polynomials $\Theta_{n}\, \left(  x\right)  =\left(  -1\right)  ^{n}H_{n}\,
\left(  x\right)  $ in terms of Hermite polynomials.

To give an another special case, we now recall that the 2-variable Hermite
Kamp\.{e} de Feriet polynomials (2VHKdFP) $H_{n}\, \left(  x,y\right)
~$\cite{AF} are defined by the generating function%
\[
\sum \limits_{n=0}^{\infty}H_{n}\, \left(  x,y\right)  \frac{t^{n}}{n!}%
=\exp \left(  xt+yt^{2}\right)  .
\]
In the special case $\alpha=\beta=e,~\psi(x)=1,~\varphi_{1}\left(  x\right)
=\varphi_{2}\left(  x\right)  =-x^{2},$ the functions given by $\left(
\ref{aa4}\right)  $ reduce to the polynomials $\Theta_{n}\, \left(  \frac
{x}{2}\right)  =H_{n}\, \left(  x,1\right)  .$

The main purpose of the present paper is to derive a generating function by
means of Cauchy's integral formula and to investigate some properties of these
functions. The set up of this paper is summarized as follows. In section 2, we
give a generating function satisfied by this family of analytic functions and
then, by using this generating function we obtain several recurrence
relations. In addition, we find differential equations satisfied by some
special families of analytic functions defined by $\left(  \ref{aa4}\right)
$, depending on choices of $\psi \left(  x\right)  $ and $\varphi_{2}\left(
x\right)  .$ Section 3 is devoted to prove a theorem to find various families
of bilateral and bilinear generating functions and then, to apply this theorem
to the special cases.

\section{A Generating Function and Recurrence Relations}

In literature, there are numerous investigations to obtain generating
functions and recurrence relations satisfied by special functions and
polynomials (see \cite{AA1,AA2,AAC,AE,AEO,ASA,EA, ES,GAA,C,KO,KYR,SOK}). It is
possible to derive a recurrence relation by using a generating function. In
this section, we first try to find a generating function for the family of
analytic functions $\Theta_{n}\, \left(  x\right)  .~$Afterwards, we give some
recurrence relations with the help of this generating function.

Now, we start with the following theorem.

\begin{theorem}
\label{theorem2.1}The family of analytic functions given by $\left(
\ref{aa4}\right)  $ has the following generating function
\begin{equation}
\sum \limits_{n=0}^{\infty}\frac{\Theta_{n}\, \left(  x\right)  }{n!}%
t^{n}=\alpha^{\varphi_{1}\left(  x\right)  }\psi \left(  x+t\right)
\beta^{-\varphi_{2}\left(  x+t\right)  } \label{aa*}%
\end{equation}
where $\varphi_{1}\left(  x\right)  ,$~$\varphi_{2}\left(  x\right)  $ and
$\psi \left(  x\right)  $ are analytic functions and $\alpha,\beta \in
\mathbb{R}^{+}\backslash \left \{  1\right \}  $.
\end{theorem}

\begin{proof}
By considering the Cauchy's integral formula%
\begin{equation}
\frac{d^{n}}{dx^{n}}\left(  \psi \left(  x\right)  \beta^{-\varphi_{2}\left(
x\right)  }\right)  =\frac{n!}{2\pi i}\oint \limits_{C}\frac{\psi \left(
z\right)  \beta^{-\varphi_{2}\left(  z\right)  }dz}{\left(  z-x\right)
^{n+1}} \label{aa1}%
\end{equation}
for a suitable contour $C$,
we can write
\begin{align*}
\sum \limits_{n=0}^{\infty}\frac{\Theta_{n}\, \left(  x\right)  }{n!}t^{n}  &
=\alpha^{\varphi_{1}\left(  x\right)  }\sum \limits_{n=0}^{\infty}\frac{t^{n}%
}{2\pi i}\oint \limits_{C}\frac{\psi \left(  z\right)  \beta^{-\varphi
_{2}\left(  z\right)  }dz}{\left(  z-x\right)  ^{n+1}}\\
&  =\frac{\alpha^{\varphi_{1}\left(  x\right)  }}{2\pi i}\oint \limits_{C}%
\frac{\psi \left(  z\right)  \beta^{-\varphi_{2}\left(  z\right)  }}{z-x}%
\sum \limits_{n=0}^{\infty}\left(  \frac{t}{z-x}\right)  ^{n}dz\\
&  =\frac{\alpha^{\varphi_{1}\left(  x\right)  }}{2\pi i}\oint \limits_{C}%
\frac{\psi \left(  z\right)  \beta^{-\varphi_{2}\left(  z\right)  }}{z-x}%
\dfrac{1}{1-\dfrac{t}{z-x}}dz\\
&  =\frac{\alpha^{\varphi_{1}\left(  x\right)  }}{2\pi i}\oint \limits_{C}%
\frac{\psi \left(  z\right)  \beta^{-\varphi_{2}\left(  z\right)  }}{z-\left(
x+t\right)  }dz
\end{align*}
for
\[
\left \vert \frac{t}{z-x}\right \vert <1.
\]
If we take into account Cauchy's integral formula $\left(  \ref{aa1}\right)  $
again, we conclude that%
\[
\sum \limits_{n=0}^{\infty}\frac{\Theta_{n}\, \left(  x\right)  }{n!}%
t^{n}=\alpha^{\varphi_{1}\left(  x\right)  }\psi \left(  x+t\right)
\beta^{-\varphi_{2}\left(  x+t\right)  }%
\]
where the point $x+t\, \ $should also be inside the contour $C$.
\end{proof}

Let's use the above theorem to obtain various recurrence relations for the
functions $\left(  \ref{aa4}\right)  $. For convenience, let the right side of
the generating function (\ref{aa*}) be denoted by%
\[
F\left(  x,t\right)  =\alpha^{\varphi_{1}\left(  x\right)  }\psi \left(
x+t\right)  \beta^{-\varphi_{2}\left(  x+t\right)  }.
\]

\begin{theorem}
\label{theorem2.2}For the family of analytic functions given by $\left(
\ref{aa4}\right)  ,$ we have
\begin{align}
&
%TCIMACRO{\dsum \limits_{p=0}^{n}}%
%BeginExpansion
{\displaystyle \sum \limits_{p=0}^{n}}
%EndExpansion
\left(
\begin{array}
[c]{c}%
n\\
p
\end{array}
\right)  \left \{  \Theta_{n-p+1}\, \left(  x\right)  \psi^{(p)}\left(
x\right)  -\Theta_{n-p}\, \left(  x\right)  \psi^{(p+1)}\left(  x\right)
\right \} \label{aa9}\\
&  =-\ln \beta~%
%TCIMACRO{\dsum \limits_{k=0}^{n}}%
%BeginExpansion
{\displaystyle \sum \limits_{k=0}^{n}}
%EndExpansion%
%TCIMACRO{\dsum \limits_{p=0}^{n-k}}%
%BeginExpansion
{\displaystyle \sum \limits_{p=0}^{n-k}}
%EndExpansion
\left(
\begin{array}
[c]{c}%
n\\
k
\end{array}
\right)  \left(
\begin{array}
[c]{c}%
n-k\\
p
\end{array}
\right)  \Theta_{n-p-k}\, \left(  x\right)  \psi^{(p)}\left(  x\right)
\varphi_{2}^{\left(  k+1\right)  }\left(  x\right) \nonumber
\end{align}
and%
\begin{align}
&
%TCIMACRO{\dsum \limits_{p=0}^{n}}%
%BeginExpansion
{\displaystyle \sum \limits_{p=0}^{n}}
%EndExpansion
\left(
\begin{array}
[c]{c}%
n\\
p
\end{array}
\right)  \left \{  \Theta_{n-p}^{\prime}\, \left(  x\right)  \psi^{(p)}\left(
x\right)  -\ln \alpha~\varphi_{1}^{\prime}\left(  x\right)  \Theta_{n-p}\,
\left(  x\right)  \psi^{(p)}\left(  x\right)  \right. \label{aa10}\\
&  \left.  -~\Theta_{n-p}\, \left(  x\right)  \psi^{(p+1)}\left(  x\right)
\right \} \nonumber \\
&  =-\ln \beta~%
%TCIMACRO{\dsum \limits_{k=0}^{n}}%
%BeginExpansion
{\displaystyle \sum \limits_{k=0}^{n}}
%EndExpansion%
%TCIMACRO{\dsum \limits_{p=0}^{n-k}}%
%BeginExpansion
{\displaystyle \sum \limits_{p=0}^{n-k}}
%EndExpansion
\left(
\begin{array}
[c]{c}%
n\\
k
\end{array}
\right)  \left(
\begin{array}
[c]{c}%
n-k\\
p
\end{array}
\right)  \Theta_{n-p-k}\, \left(  x\right)  \psi^{(p)}\left(  x\right)
\varphi_{2}^{\left(  k+1\right)  }\left(  x\right) \nonumber
\end{align}
for $n\geq0.$
\end{theorem}

\begin{proof}
If we take the derivative of the function $F\left(  x,t\right)  =\alpha
^{\varphi_{1}\left(  x\right)  }\psi \left(  x+t\right)  \beta^{-\varphi
_{2}\left(  x+t\right)  }$ with respect to $t,\,$we arrive at%
\begin{equation}
\psi \left(  x+t\right)  \frac{\partial}{\partial t}F\left(  x,t\right)
=\left \{  -\ln \beta~\varphi_{2}^{^{\prime}}\left(  x+t\right)  \psi \left(
x+t\right)  +\psi^{\prime}\left(  x+t\right)  \right \}  F\left(  x,t\right)  .
\label{aa7}%
\end{equation}
Considering the left side of the generating function (\ref{aa*}) in $\left(
\ref{aa7}\right)  $ gives
\[
\psi \left(  x+t\right)  \sum \limits_{n=0}^{\infty}\Theta_{n+1}\, \left(
x\right)  \frac{t^{n}}{n!}=\left \{  -\ln \beta~\varphi_{2}^{^{\prime}}\left(
x+t\right)  \psi \left(  x+t\right)  +\psi^{\prime}\left(  x+t\right)
\right \}  \sum \limits_{n=0}^{\infty}\Theta_{n}\, \left(  x\right)  \frac
{t^{n}}{n!}.
\]
If we use taylor series of the analytic functions $\psi \left(  x+t\right)
,~\psi^{\prime}\left(  x+t\right)  $ and $\varphi_{2}^{^{\prime}}\left(
x+t\right)  $ at $t=0,$ respectively%
\begin{align*}
\psi \left(  x+t\right)   &  =\sum \limits_{p=0}^{\infty}\psi^{\left(  p\right)
}\left(  x\right)  \dfrac{t^{p}}{p!},\\
\psi^{\prime}\left(  x+t\right)   &  =\sum \limits_{p=0}^{\infty}\psi^{\left(
p+1\right)  }\left(  x\right)  \frac{t^{p}}{p!}%
\end{align*}
and%
\[
\varphi_{2}^{^{\prime}}\left(  x+t\right)  =\sum \limits_{k=0}^{\infty}%
\varphi_{2}^{\left(  k+1\right)  }\left(  x\right)  \frac{t^{k}}{k!},
\]
we have%
\begin{align*}
\sum \limits_{n,p=0}^{\infty}\Theta_{n+1}\, \left(  x\right)  \psi^{\left(
p\right)  }\left(  x\right)  \frac{t^{n+p}}{n!p!}  &  =-\ln \beta
~\sum \limits_{n,k,p=0}^{\infty}\Theta_{n}\, \left(  x\right)  \varphi
_{2}^{\left(  k+1\right)  }\left(  x\right)  \psi^{\left(  p\right)  }\left(
x\right)  \frac{t^{n+k+p}}{n!k!p!}\\
&  +\sum \limits_{n,p=0}^{\infty}\Theta_{n}\, \left(  x\right)  \psi^{\left(
p+1\right)  }\left(  x\right)  \frac{t^{n+p}}{n!p!}.
\end{align*}
Upon inverting the order of summation above, if we replace $n$ by $n-p,$ we
can write%
\begin{align*}
&  \sum \limits_{n=0}^{\infty}\sum \limits_{p=0}^{n}\Theta_{n-p+1}\, \left(
x\right)  \psi^{\left(  p\right)  }\left(  x\right)  \frac{t^{n}}{\left(
n-p\right)  !p!}\\
&  =-\ln \beta~\sum \limits_{n,k=0}^{\infty}\sum \limits_{p=0}^{n}\Theta_{n-p}\,
\left(  x\right)  \varphi_{2}^{\left(  k+1\right)  }\left(  x\right)
\psi^{\left(  p\right)  }\left(  x\right)  \frac{t^{n+k}}{\left(  n-p\right)
!k!p!}\\
&  +\sum \limits_{n=0}^{\infty}\sum \limits_{p=0}^{n}\Theta_{n-p}\, \left(
x\right)  \psi^{\left(  p+1\right)  }\left(  x\right)  \frac{t^{n}}{\left(
n-p\right)  !p!}.
\end{align*}
If we take $n-k$ instead of $n$ in the first summation in the right side of
this equation and then compare the coefficients of $\dfrac{t^{n}}{n!},$ we
complete the proof of $\left(  \ref{aa9}\right)  $ .

On the other hand, it is easily seen that $F(x,t)$~satisfies%
\[
\psi \left(  x+t\right)  \frac{\partial}{\partial x}F\left(  x,t\right)
=\left \{  \ln \alpha~\varphi_{1}^{\prime}\left(  x\right)  \psi \left(
x+t\right)  -\ln \beta~\varphi_{2}^{^{\prime}}\left(  x+t\right)  \psi \left(
x+t\right)  +\psi^{\prime}\left(  x+t\right)  \right \}  F\left(  x,t\right)
.
\]
In order to obtain $\left(  \ref{aa10}\right)  $, it is enough to make similar
calculations above.
\end{proof}

\begin{corollary}
\label{corollary2.1}Combining the recurrence relations in Theorem
\ref{theorem2.2} gives%
\[%
\begin{array}
[c]{c}%
%TCIMACRO{\dsum \limits_{p=0}^{n}}%
%BeginExpansion
{\displaystyle \sum \limits_{p=0}^{n}}
%EndExpansion
\left \{  \Theta_{n-p+1}\, \left(  x\right)  \psi^{(p)}\left(  x\right)
-\Theta_{n-p}\, \left(  x\right)  \psi^{(p+1)}\left(  x\right)  -\Theta
_{n-p}^{\prime}\, \left(  x\right)  \psi^{(p)}\left(  x\right)  \right. \\
\left.  +\ln \alpha~\varphi_{1}^{\prime}\left(  x\right)  \Theta_{n-p}\,
\left(  x\right)  \psi^{(p)}\left(  x\right)  +~\Theta_{n-p}\, \left(
x\right)  \psi^{(p+1)}\left(  x\right)  \right \}  \left(
\begin{array}
[c]{c}%
n\\
p
\end{array}
\right)  =0
\end{array}
\]
for $n\geq0.$
\end{corollary}

\begin{theorem}
\label{theorem2.3}The family of analytic functions given by $\left(
\ref{aa4}\right)  $ satisfies%
\[
\Theta_{n}^{\prime}\, \left(  x\right)  =\Theta_{n+1}\, \left(  x\right)
+\ln \alpha~\varphi_{1}^{\prime}\left(  x\right)  \Theta_{n}\, \left(
x\right)
\]
for $n\geq0.$
\end{theorem}

\begin{proof}
For the function $F\left(  x,t\right)  =\alpha^{\varphi_{1}\left(  x\right)
}\psi \left(  x+t\right)  \beta^{-\varphi_{2}\left(  x+t\right)  },$ the
following equation holds%
\begin{equation}
\frac{\partial}{\partial x}F\left(  x,t\right)  =\frac{\partial}{\partial
t}F\left(  x,t\right)  +\ln \alpha \text{ }\varphi_{1}^{\prime}\left(  x\right)
F\left(  x,t\right)  . \label{aa6}%
\end{equation}
If we take into consideration generating function (\ref{aa*}) in $\left(
\ref{aa6}\right)  ,$ we can write%
\[%
%TCIMACRO{\dsum \limits_{n=0}^{\infty}}%
%BeginExpansion
{\displaystyle \sum \limits_{n=0}^{\infty}}
%EndExpansion
\Theta_{n}^{\prime}\, \left(  x\right)  \frac{t^{n}}{n!}=%
%TCIMACRO{\dsum \limits_{n=0}^{\infty}}%
%BeginExpansion
{\displaystyle \sum \limits_{n=0}^{\infty}}
%EndExpansion
\Theta_{n+1}\, \left(  x\right)  \frac{t^{n}}{n!}+\ln \alpha \text{ }\varphi
_{1}^{\prime}\left(  x\right)
%TCIMACRO{\dsum \limits_{n=0}^{\infty}}%
%BeginExpansion
{\displaystyle \sum \limits_{n=0}^{\infty}}
%EndExpansion
\Theta_{n}\, \left(  x\right)  \frac{t^{n}}{n!}.
\]
By equating the coefficients of $\dfrac{t^{n}}{n!}$, it follows that%
\[
\Theta_{n}^{\prime}\, \left(  x\right)  =\Theta_{n+1}\, \left(  x\right)
+\ln \alpha~\varphi_{1}^{\prime}\left(  x\right)  \Theta_{n}\, \left(
x\right)
\]
for $n\geq0.$
\end{proof}

\begin{theorem}
\label{theorem2.4}An another recurrence relation for the analytic functions
$\Theta_{n}\, \left(  x\right)  $ is as follows%
\begin{equation}%
\begin{array}
[c]{l}%
%TCIMACRO{\dsum \limits_{p=0}^{n}}%
%BeginExpansion
{\displaystyle \sum \limits_{p=0}^{n}}
%EndExpansion
\left \{  \ln \alpha~\Theta_{n-p+1}\, \left(  x\right)  \psi^{(p)}\left(
x\right)  \varphi_{1}^{\prime}\left(  x\right)  +\Theta_{n-p+1}\, \left(
x\right)  \psi^{(p+1)}\left(  x\right)  \right. \\
\\
\left.  -\Theta_{n-p}^{\prime}\, \left(  x\right)  \psi^{(p+1)}\left(
x\right)  \right \}  \left(
\begin{array}
[c]{c}%
n\\
p
\end{array}
\right) \\
\\
+\ln \beta%
%TCIMACRO{\dsum \limits_{k=0}^{n}}%
%BeginExpansion
{\displaystyle \sum \limits_{k=0}^{n}}
%EndExpansion%
%TCIMACRO{\dsum \limits_{p=0}^{n-k}}%
%BeginExpansion
{\displaystyle \sum \limits_{p=0}^{n-k}}
%EndExpansion
\left \{  \left(
\begin{array}
[c]{c}%
n\\
k
\end{array}
\right)  \left(
\begin{array}
[c]{c}%
n-k\\
p
\end{array}
\right)  \Theta_{n-k-p}^{\prime}\, \left(  x\right)  \varphi_{2}%
^{(k+1)}\left(  x\right)  \psi^{(p)}\left(  x\right)  \right. \\
\\
\left.  -~\left(
\begin{array}
[c]{c}%
n\\
k
\end{array}
\right)  \left(
\begin{array}
[c]{c}%
n-k\\
p
\end{array}
\right)  \Theta_{n-k-p+1}\, \left(  x\right)  \psi^{(p)}\left(  x\right)
\varphi_{2}^{(k+1)}\left(  x\right)  \right \}  =0
\end{array}
\label{aa11}%
\end{equation}
for $n\geq0.$
\end{theorem}

\begin{proof}
It is obvious from the generating function (\ref{aa*}) that%
\begin{align*}
&  \left \{  \ln \alpha~\varphi_{1}^{\prime}\left(  x\right)  \psi \left(
x+t\right)  -\ln \beta~\psi \left(  x+t\right)  \varphi_{2}^{\prime}\left(
x+t\right)  +\psi^{\prime}\left(  x+t\right)  \right \}  \frac{\partial
}{\partial t}F\left(  x,t\right) \\
&  =\left \{  \psi^{\prime}\left(  x+t\right)  -\ln \beta~\psi \left(
x+t\right)  \varphi_{2}^{\prime}\left(  x+t\right)  \right \}  \frac{\partial
}{\partial x}F\left(  x,t\right)
\end{align*}
from which, we can obtain the desired relation.
\end{proof}

\begin{corollary}
\label{corollary2.2}If we combine Theorem \ref{theorem2.3} and Theorem
\ref{theorem2.4} for the special case $\psi \left(  x\right)  =1,$ we obtain%
\[
\Theta_{n+1}\, \left(  x\right)  =-\ln \beta~\sum \limits_{k=0}^{n}\left(
\begin{array}
[c]{c}%
n\\
k
\end{array}
\right)  \varphi_{2}^{(k+1)}\left(  x\right)  \Theta_{n-k}\, \left(  x\right)
\]
for $n\geq0.$
\end{corollary}

Now, by means of the recurrence relations given above, we\ show that some
special cases of the analytic functions given by $\left(  \ref{aa4}\right)  $
are solutions of a differential equation.

Let $\psi \left(  x\right)  =1.$ Then from $\left(  \ref{aa4}\right)  ,$\ we
have
\begin{equation}
\Theta_{n}\, \left(  x\right)  =\alpha^{\varphi_{1}\left(  x\right)  }%
\frac{d^{n}}{dx^{n}}\beta^{-\varphi_{2}\left(  x\right)  }, \label{4}%
\end{equation}
which verifies the generating function
\[
\sum \limits_{n=0}^{\infty}\frac{\Theta_{n}\, \left(  x\right)  }{n!}%
t^{n}=\alpha^{\varphi_{1}\left(  x\right)  }\beta^{-\varphi_{2}\left(
x+t\right)  }.
\]

\begin{theorem}
\label{theorem2.55}Let $\varphi_{2}\left(  x\right)  $ be a polynomial of
degree $2$ and $\varphi_{1}\left(  x\right)  ,$ not constant, be an analytic
function. For the functions $y=\Theta_{n}\, \left(  x\right)  $ defined by
$\left(  \ref{4}\right)  ,$ one easily gets%
\[%
\begin{array}
[c]{l}%
y^{\prime \prime}+\left \{  \ln \beta~\varphi_{2}^{\prime}\left(  x\right)
-2\ln \alpha~\varphi_{1}^{\prime}\left(  x\right)  \right \}  y^{\prime
}+\left \{  \ln^{2}\alpha~\left(  \varphi_{1}^{\prime}\left(  x\right)
\right)  ^{2}-\ln \alpha~\varphi_{1}^{\prime \prime}\left(  x\right)  \right. \\
\\
\left.  -\ln \alpha~\ln \beta~\varphi_{1}^{\prime}\left(  x\right)  \varphi
_{2}^{\prime}\left(  x\right)  +\left(  n+1\right)  \ln \beta~\varphi
_{2}^{\prime \prime}\left(  x\right)  \right \}  y=0.
\end{array}
\]

\end{theorem}

\begin{proof}
To prove this theorem, it is enough to combine Theorem \ref{theorem2.3} and
Theorem \ref{theorem2.4}.
\end{proof}

\begin{example}
In the special case $\varphi_{1}\left(  x\right)  =\varphi_{2}\left(
x\right)  =x^{2}$,~$\alpha=\beta=e,$ it is obvious that $\Theta_{n}\, \left(
x\right)  =\left(  -1\right)  ^{n}H_{n}\, \left(  x\right)  $ satisfies
Hermite differential equation%
\[
y^{\prime \prime}-2xy^{\prime}+2ny=0.
\]

\end{example}

\begin{theorem}
\label{theorem2.5}Assume that $\varphi_{2}\left(  x\right)  $ is a polynomial
of degree $3$ and $\varphi_{1}\left(  x\right)  $ is an analytic function
which is not constant. Then the family of functions $y=\Theta_{n}\, \left(
x\right)  $ satisfies the third order linear differential equation%

\[%
\begin{array}
[c]{l}%
y^{\prime \prime \prime}+\left \{  \ln \beta~\varphi_{2}^{\prime}\left(  x\right)
-3\ln \alpha~\varphi_{1}^{\prime}\left(  x\right)  \right \}  y^{\prime \prime
}+\left \{  \left(  n+2\right)  \ln \beta~\varphi_{2}^{\prime \prime}\left(
x\right)  -3\ln \alpha~\varphi_{1}^{\prime \prime}\left(  x\right)  \right. \\
\\
\left.  -2\ln \alpha~\ln \beta~\varphi_{1}^{\prime}\left(  x\right)  \varphi
_{2}^{\prime}\left(  x\right)  +3\ln^{2}\alpha~\left(  \varphi_{1}^{\prime
}\left(  x\right)  \right)  ^{2}\right \}  y^{\prime}+\left \{  -\ln^{3}%
\alpha~\left(  \varphi_{1}^{\prime}\left(  x\right)  \right)  ^{3}\right. \\
\\
\left.  +\ln^{2}\alpha \ln \beta~\varphi_{2}^{\prime}\left(  x\right)  \left(
\varphi_{1}^{\prime}\left(  x\right)  \right)  ^{2}+3\ln^{2}\alpha~\varphi
_{1}^{\prime}\left(  x\right)  \varphi_{1}^{\prime \prime}\left(  x\right)
-\ln \alpha~\varphi_{1}^{\prime \prime \prime}\left(  x\right)  \right. \\
\\
\left.  -\ln \alpha~\ln \beta~\varphi_{1}^{\prime \prime}\left(  x\right)
\varphi_{2}^{\prime}\left(  x\right)  -\left(  n+2\right)  \ln \alpha~\ln
\beta~\varphi_{2}^{\prime \prime}\left(  x\right)  \varphi_{1}^{\prime}\left(
x\right)  \right. \\
\\
\left.  +\frac{1}{2}\left(  n+1\right)  \left(  n+2\right)  \ln \beta
~\varphi_{2}^{\prime \prime \prime}\left(  x\right)  \right \}  y=0.\\
\end{array}
\]

\end{theorem}

\begin{proof}
As $\psi \left(  x\right)  =1$ and $\varphi_{2}\left(  x\right)  $ is a
polynomial of degree $3,$ then the relation given by $\left(  \ref{aa11}%
\right)  $ reduces to%
\[%
\begin{array}
[c]{l}%
\ln \alpha~\Theta_{n+1}\, \left(  x\right)  \varphi_{1}^{\prime}\left(
x\right)  +\ln \beta \left \{  \varphi_{2}^{\prime}\left(  x\right)  \left(
\Theta_{n}^{\prime}\, \left(  x\right)  -~\Theta_{n+1}\, \left(  x\right)
\right)  \right. \\
\\
\left.  +\varphi_{2}^{\prime \prime}\left(  x\right)  \left(  \Theta
_{n-1}^{\prime}\, \left(  x\right)  -~\Theta_{n}\, \left(  x\right)  \right)
n\right. \\
\\
\left.  +\frac{1}{2}n\left(  n-1\right)  \varphi_{2}^{\prime \prime \prime
}\left(  x\right)  \left(  \Theta_{n-2}^{\prime}\, \left(  x\right)
-~\Theta_{n-1}\, \left(  x\right)  \right)  \right \}  =0
\end{array}
\]
Replacing $n$ by $n+2$, \ it follows that%
\begin{equation}%
\begin{array}
[c]{l}%
\ln \alpha~\Theta_{n+3}\, \left(  x\right)  \varphi_{1}^{\prime}\left(
x\right)  +\ln \beta \left \{  \varphi_{2}^{\prime}\left(  x\right)  \left(
\Theta_{n+2}^{\prime}\, \left(  x\right)  -~\Theta_{n+3}\, \left(  x\right)
\right)  \right. \\
\\
\left.  +\varphi_{2}^{\prime \prime}\left(  x\right)  \left(  \Theta
_{n+1}^{\prime}\, \left(  x\right)  -~\Theta_{n+2}\, \left(  x\right)
\right)  \left(  n+2\right)  \right. \\
\\
\left.  +\frac{1}{2}\left(  n+1\right)  \left(  n+2\right)  \varphi
_{2}^{\prime \prime \prime}\left(  x\right)  \left(  \Theta_{n}^{\prime}\,
\left(  x\right)  -~\Theta_{n+1}\, \left(  x\right)  \right)  \right \}  =0.
\end{array}
\label{aa12}%
\end{equation}
On the other hand, we have the following recurrence relation from Theorem
\ref{theorem2.3}%
\[
\Theta_{n+1}\, \left(  x\right)  =\Theta_{n}^{\prime}\, \left(  x\right)
-\ln \alpha~\varphi_{1}^{\prime}\left(  x\right)  \Theta_{n}\, \left(
x\right)  ,
\]
from which, by repeating this recurrence relation, we can obtain%

\begin{align*}
\Theta_{n+1}^{\prime}\, \left(  x\right)   &  =\Theta_{n}^{\prime \prime}\,
\left(  x\right)  -\ln \alpha~\varphi_{1}^{\prime}\left(  x\right)  \Theta
_{n}^{\prime}\, \left(  x\right)  -\ln \alpha~\varphi_{1}^{\prime \prime}\left(
x\right)  \Theta_{n}\, \left(  x\right)  ,\\
& \\
\Theta_{n+2}\, \left(  x\right)   &  =\Theta_{n}^{\prime \prime}\, \left(
x\right)  -2\ln \alpha~\varphi_{1}^{\prime}\left(  x\right)  \Theta_{n}%
^{\prime}\, \left(  x\right)  +\left \{  \ln^{2}\alpha \left(  \varphi
_{1}^{\prime}\left(  x\right)  \right)  ^{2}-\ln \alpha~\varphi_{1}%
^{\prime \prime}\left(  x\right)  \right \}  \Theta_{n}\, \left(  x\right)  ,\\
& \\
\Theta_{n+2}^{\prime}\, \left(  x\right)   &  =\Theta_{n}^{\prime \prime \prime
}\, \left(  x\right)  -2\ln \alpha~\varphi_{1}^{\prime}\left(  x\right)
\Theta_{n}^{\prime \prime}\, \left(  x\right)  +\left \{  \ln^{2}\alpha \left(
\varphi_{1}^{\prime}\left(  x\right)  \right)  ^{2}-3\ln \alpha~\varphi
_{1}^{\prime \prime}\left(  x\right)  \right \}  \Theta_{n}^{\prime}\, \left(
x\right) \\
&  +\left \{  2\ln^{2}\alpha~\varphi_{1}^{\prime}\left(  x\right)  \varphi
_{1}^{\prime \prime}\left(  x\right)  -\ln \alpha~\varphi_{1}^{\prime
\prime \prime}\left(  x\right)  \right \}  \Theta_{n}\, \left(  x\right)  ,\\
& \\
\Theta_{n+3}\, \left(  x\right)   &  =\Theta_{n}^{\prime \prime \prime}\,
\left(  x\right)  -3\ln \alpha~\varphi_{1}^{\prime}\left(  x\right)  \Theta
_{n}^{\prime \prime}\, \left(  x\right)  +\left \{  3\ln^{2}\alpha~\left(
\varphi_{1}^{\prime}\left(  x\right)  \right)  ^{2}-3\ln \alpha~\varphi
_{1}^{\prime \prime}\left(  x\right)  \right \}  \Theta_{n}^{\prime}\, \left(
x\right) \\
&  +\left \{  -\ln^{3}\alpha~\left(  \varphi_{1}^{\prime}\left(  x\right)
\right)  ^{3}+3\ln^{2}\alpha~\varphi_{1}^{\prime}\left(  x\right)  \varphi
_{1}^{\prime \prime}\left(  x\right)  -\ln \alpha~\varphi_{1}^{\prime
\prime \prime}\left(  x\right)  \right \}  \Theta_{n}\, \left(  x\right)  .
\end{align*}
Taking into account these equalities in $\left(  \ref{aa12}\right)  $ leads to
the desired third order linear differential equation .
\end{proof}

Similarly, as a result of Theorem \ref{theorem2.3} and Theorem
\ref{theorem2.4}, it is possible to give the next theorem.

\begin{theorem}
\label{theorem2.6}Let $\varphi_{2}\left(  x\right)  $ be a polynomial of
degree $4$ and $\varphi_{1}\left(  x\right)  ,$ not constant, be an analytic
function. The functions$\ y=\Theta_{n}\, \left(  x\right)  $ are solutions of
the fourth order linear differential equation%
\[%
\begin{array}
[c]{l}%
y^{\left(  iv\right)  }+\left \{  \ln \beta~\varphi_{2}^{\prime}\left(
x\right)  -4\ln \alpha~\varphi_{1}^{\prime}\left(  x\right)  \right \}
y^{\prime \prime \prime}+\left \{  6\ln^{2}\alpha~\left(  \varphi_{1}^{\prime
}\left(  x\right)  \right)  ^{2}\right. \\
\\
\left.  -3\ln \alpha~\ln \beta~\varphi_{1}^{\prime}\left(  x\right)  \varphi
_{2}^{\prime}\left(  x\right)  -6\ln \alpha~\varphi_{1}^{\prime \prime}\left(
x\right)  +\left(  n+3\right)  \ln \beta~\varphi_{2}^{\prime \prime}\left(
x\right)  \right \}  y^{\prime \prime}\\
\\
+\left \{  -4\ln^{3}\alpha~\left(  \varphi_{1}^{\prime}\left(  x\right)
\right)  ^{3}+3\ln^{2}\alpha \ln \beta~\varphi_{2}^{\prime}\left(  x\right)
\left(  \varphi_{1}^{\prime}\left(  x\right)  \right)  ^{2}+12\ln^{2}%
\alpha~\varphi_{1}^{\prime}\left(  x\right)  \varphi_{1}^{\prime \prime}\left(
x\right)  \right. \\
\\
\left.  -4\ln \alpha~\varphi_{1}^{\prime \prime \prime}\left(  x\right)
-3\ln \alpha~\ln \beta~\varphi_{1}^{\prime \prime}\left(  x\right)  \varphi
_{2}^{\prime}\left(  x\right)  -2\left(  n+3\right)  \ln \alpha~\ln
\beta~\varphi_{2}^{\prime \prime}\left(  x\right)  \varphi_{1}^{\prime}\left(
x\right)  \right. \\
\\
\left.  +\frac{1}{2}\left(  n+2\right)  \left(  n+3\right)  \ln \beta
~\varphi_{2}^{\prime \prime \prime}\left(  x\right)  \right \}  y^{\prime
}+\left \{  \ln^{4}\alpha~\left(  \varphi_{1}^{\prime}\left(  x\right)
\right)  ^{4}-\ln^{3}\alpha \ln \beta~\left(  \varphi_{1}^{\prime}\left(
x\right)  \right)  ^{3}\varphi_{2}^{\prime}\left(  x\right)  \right. \\
\\
\left.  -6\ln^{3}\alpha~\left(  \varphi_{1}^{\prime}\left(  x\right)  \right)
^{2}\varphi_{1}^{\prime \prime}\left(  x\right)  +4\ln^{2}\alpha~\varphi
_{1}^{\prime}\left(  x\right)  \varphi_{1}^{\prime \prime \prime}\left(
x\right)  +3\ln^{2}\alpha \ln \beta~\varphi_{1}^{\prime}\left(  x\right)
\varphi_{1}^{\prime \prime}\left(  x\right)  \varphi_{2}^{\prime}\left(
x\right)  \right. \\
\\
\left.  -\ln \alpha~\ln \beta~\varphi_{1}^{\prime \prime \prime}\left(  x\right)
\varphi_{2}^{\prime}\left(  x\right)  -3\ln^{2}\alpha~\left(  \varphi
_{1}^{\prime \prime}\left(  x\right)  \right)  ^{2}-\ln \alpha~\varphi
_{1}^{\left(  iv\right)  }\left(  x\right)  \right. \\
\\
\left.  +\left(  n+3\right)  \ln^{2}\alpha \ln \beta \left(  \varphi_{1}^{\prime
}\left(  x\right)  \right)  ^{2}\varphi_{2}^{\prime \prime}\left(  x\right)
-\left(  n+3\right)  \ln \alpha~\ln \beta~\varphi_{1}^{\prime \prime}\left(
x\right)  \varphi_{2}^{\prime \prime}\left(  x\right)  \right. \\
\\
\left.  -\frac{1}{2}\left(  n+2\right)  \left(  n+3\right)  \ln \alpha~\ln
\beta~\varphi_{1}^{\prime}\left(  x\right)  \varphi_{2}^{\prime \prime \prime
}\left(  x\right)  +\frac{1}{6}\left(  n+1\right)  \left(  n+2\right)  \left(
n+3\right)  \ln \beta~\varphi_{2}^{\left(  iv\right)  }\left(  x\right)
\right \}  y=0
\end{array}
\]

\end{theorem}

\begin{example}
For example, taking $\varphi_{1}\left(  x\right)  =\varphi_{2}\left(
x\right)  =-x^{4}$ ,~$\alpha=\beta=e$ in the equation $\left(  \ref{4}\right)
,$ we have
\[
\Theta_{n}\, \left(  x\right)  =e^{-x^{4}}\frac{d^{n}}{dx^{n}}\left(
e^{x^{4}}\right)  \, \, \, \, \, \,,\, \, \, \, \, \, \, \, \,n=0,1,2...,
\]
which gives a family of polynomials of degree $3n.$ Theorem \ref{theorem2.6}
shows that these polynomials are solutions of the following fourth order
linear differential equation%
\[%
\begin{array}
[c]{l}%
\,y^{\left(  4\right)  }+12x^{3}y^{\prime \prime \prime}+\left \{  48x^{6}%
-12\left(  n-3\right)  x^{2}\right \}  y^{\prime \prime}\\
\\
+\left \{  64x^{9}+\left(  144-96n\right)  x^{5}-12\left(  n^{2}+5n-2\right)
x\right \}  y^{\prime}\\
\\
+\left \{  -192nx^{8}-48\left(  n^{2}+8n\right)  x^{4}-4n\left(  n^{2}%
+6n+11\right)  \right \}  y=0.
\end{array}
\]

\begin{remark}
We observe that if $\varphi_{2}\left(  x\right)  $ is a polynomial of degree
$m$ $(m\leq n)$ and $\varphi_{1}\left(  x\right)  $ is an analytic function,
not constant, then the functions$\ y=\Theta_{n}\, \left(  x\right)  $ satisfy
the $m$-th order linear differential equation. But, it is complicated to
obtain its explicit form.
\end{remark}
\end{example}

\section{Bilinear and Bilateral Generating Functions}

In this section, we \ try to derive many families of bilinear and bilateral
generating functions for the family of analytic functions given by (\ref{aa4})
by means of the similar method presented in \cite{AD,AA2,AAC,AE,AEO,ASA,EA,
ES,GAA, C,KO,SOK}.

For this purpose, let's begin the following theorem.

\begin{theorem}
\label{theorem3.1}Corresponding to an identically non-vanishing function
$\Omega_{\mu}(y_{1},...,y_{s}\,)$ of $s$ complex variables $y_{1},...,y_{s}$
$(s\in \mathbb{N})$ and of complex order $\mu$, if
\begin{equation}
\Lambda_{\mu,\nu}(y_{1},...,y_{s};z):=\sum \limits_{k=0}^{\infty}a_{k}%
\Omega_{\mu+\nu k}(y_{1},...,y_{s}\,)z^{k} \label{g1}%
\end{equation}%
\[
(a_{k}\neq0\,,\, \, \, \mu,\nu \in \mathbb{C})
\]
and
\begin{equation}
\Phi_{n,p,\mu,\nu}(x;y_{1},...,y_{s};\zeta):=\sum \limits_{k=0}^{[n/p]}%
\frac{a_{k}}{\left(  n-pk\right)  !}\Theta_{n-pk}(x)\Omega_{\mu+\nu k}%
(y_{1},...,y_{s}\,)\zeta^{k} \label{g2}%
\end{equation}%
\[
\left(  n,p\in \mathbb{N}\right)
\]
then it follows
\begin{equation}
\sum \limits_{n=0}^{\infty}\Phi_{n,p,\mu,\nu}(x;y_{1},...,y_{s};\frac{\eta
}{t^{p}})t^{n}=\alpha^{\varphi_{1}\left(  x\right)  }\psi \left(  x+t\right)
\beta^{-\varphi_{2}\left(  x+t\right)  }\Lambda_{\mu,\nu}(y_{1},...,y_{s}%
;\eta) \label{g3}%
\end{equation}
provided that each member of (\ref{g3}) exists.
\end{theorem}

\begin{proof}
For convenience, let $S$ denote the left side of (\ref{g3}) in Theorem
\ref{theorem3.1}. If we substitute the explicit form of polynomials
\[
\Phi_{n,p,\mu,\nu}(x;y_{1},...,y_{s};\frac{\eta}{t^{p}})
\]
from the definition (\ref{g2}) in (\ref{g3}), we can write
\begin{equation}
S=\sum \limits_{n=0}^{\infty}\sum \limits_{k=0}^{[n/p]}\frac{a_{k}}{\left(
n-pk\right)  !}\Theta_{n-pk}(x)\Omega_{\mu+\nu k}(y_{1},...,y_{s}\,)\eta
^{k}t^{n-pk}\,. \label{g4}%
\end{equation}
Upon inverting the order of summation in (\ref{g4}), if we replace $n$ by
$n+pk,$ it follows that
\begin{align*}
S  &  =\sum \limits_{n=0}^{\infty}\sum \limits_{k=0}^{\infty}\frac{a_{k}}%
{n!}\Theta_{n}(x)\Omega_{\mu+\nu k}(y_{1},...,y_{s}\,)\eta^{k}t^{n}\\
&  =\sum_{n=0}^{\infty}\Theta_{n}(x)\frac{t^{n}}{n!}\sum \limits_{k=0}^{\infty
}a_{k}\Omega_{\mu+\nu k}(y_{1},...,y_{s}\,)\eta^{k}\\
&  =\alpha^{\varphi_{1}\left(  x\right)  }\psi \left(  x+t\right)
\beta^{-\varphi_{2}\left(  x+t\right)  }\Lambda_{\mu,\nu}(y_{1},...,y_{s}%
;\eta),
\end{align*}
which completes the proof.
\end{proof}

The above theorem can be used to construct various families of bilateral and
bilateral generating functions by expressing the multivariable function
\[
\Omega_{\mu+\nu k}(y_{1},...,y_{s}\,)\, \,(k\in \mathbb{N}_{0}\,,\,
\,s\in \mathbb{N)}%
\]
in terms of simpler function of one and more variables.

For example, taking%
\[
s=1\  \text{\thinspace and \ }\Omega_{\mu+\nu k}(y)=\mathcal{B}_{\mu+\nu
k}^{(\alpha)}(y;\lambda),
\]
where $\mathcal{B}_{n}^{(\alpha)}(x;\lambda)$ ($\lambda \in \mathbb{C)~}$denotes
Apostol-Bernoulli polynomials of order $\alpha \in \mathbb{N}_{0}$ which are
defined by the generating function \cite{L-S}%
\begin{equation}%
\begin{array}
[c]{c}%
\left(  \dfrac{t}{\lambda e^{t}-1}\right)  ^{\alpha}e^{xt}=\sum \limits_{n=0}%
^{\infty}\mathcal{B}_{n}^{(\alpha)}(x;\lambda)\dfrac{t^{n}}{n!}\\
\\
\left(  \left \vert t\right \vert <2\pi,\text{ when }\lambda=1~~;~~\left \vert
t\right \vert <\left \vert \log \lambda \right \vert ,\text{ when }\lambda
\neq1\right)  ,
\end{array}
\label{c1}%
\end{equation}
leads to a class of bilateral generating functions for Apostol-Bernoulli
polynomials and the family of analytic functions defined by (\ref{aa4}).

\begin{corollary}
\label{corollary3.1} If $\Lambda_{\mu,\nu}(y;z):=\sum \limits_{k=0}^{\infty
}a_{k}\mathcal{B}_{\mu+\nu k}^{(\alpha)}(y;\lambda)~z^{k}$ where $a_{k}%
\neq0\,,\, \, \, \nu,\mu \in \mathbb{C};$ and
\begin{align*}
&  \Phi_{n,p,\mu,\nu}(x;y;\zeta)\\
&  :=\sum \limits_{k=0}^{[n/p]}\frac{a_{k}}{\left(  n-pk\right)  !}%
\Theta_{n-pk}(x)\mathcal{B}_{\mu+\nu k}^{(\alpha)}(y;\lambda)\zeta^{k}%
\end{align*}
where $n,p\in \mathbb{N}$, then we have%
\begin{equation}
\sum \limits_{n=0}^{\infty}\Theta_{n,p,\mu,\nu}\left(  x;y;\dfrac{\eta}{t^{p}%
}\right)  t^{n}=\alpha^{\varphi_{1}\left(  x\right)  }\psi \left(  x+t\right)
\beta^{-\varphi_{2}\left(  x+t\right)  }\Lambda_{\mu,\nu}(y;\eta) \label{h1}%
\end{equation}
provided that each member of (\ref{h1}) exists.
\end{corollary}

\begin{remark}
\label{remark3.1} If we use the generating relation (\ref{c1}) for
Apostol-Bernoulli polynomials by taking $a_{k}=\dfrac{1}{k!},$ $\mu=0,$
$\nu=1$, we get
\begin{align*}
&  \sum \limits_{n=0}^{\infty}\sum \limits_{k=0}^{[n/p]}\Theta_{n-pk}%
(x)\mathcal{B}_{k}^{(\alpha)}(y;\lambda)\frac{\eta^{k}}{k!}\frac{t^{n-pk}%
}{\left(  n-pk\right)  !}\\
&  =\alpha^{\varphi_{1}\left(  x\right)  }\psi \left(  x+t\right)
\beta^{-\varphi_{2}\left(  x+t\right)  }\left(  \dfrac{\eta}{\lambda e^{\eta
}-1}\right)  ^{\alpha}e^{y\eta},
\end{align*}
where
\[
\left(  \left \vert \eta \right \vert <2\pi,\text{ when }\lambda=1~~;~~\left \vert
\eta \right \vert <\left \vert \log \lambda \right \vert ,\text{ when }\lambda
\neq1\right)  .
\]

\end{remark}

In a similar manner, choosing $s=1\, \, \,$and $\Omega_{\mu+\nu k}%
(y\,)=\Theta_{\mu+\nu k}(y)$ in Theorem \ref{theorem3.1}, we obtain the
following class of bilinear generating functions for the functions generated
by (\ref{aa*}).

\begin{corollary}
\label{corollary3.2}If
\[
\Lambda_{\mu,\nu}(y;z):=%
%TCIMACRO{\dsum \limits_{k=0}^{\infty}}%
%BeginExpansion
{\displaystyle \sum \limits_{k=0}^{\infty}}
%EndExpansion
a_{k}\Theta_{\mu+\nu k}(y)z^{k},
\]
where $a_{k}\neq0\,,\, \, \mu,\nu \in \mathbb{C}$ and
\begin{align*}
&  \Phi_{n,p,\mu,\nu}(x;y;\zeta)\\
&  :=\sum \limits_{k=0}^{[n/p]}a_{k}\Theta_{\mu+\nu k}(y)\frac{\Theta
_{n-pk}(x)}{\left(  n-pk\right)  !}\zeta^{k}%
\end{align*}
where $n,p\in \mathbb{N}.$ Then we have
\begin{equation}%
\begin{array}
[c]{ll}%
\sum \limits_{n=0}^{\infty}\Phi_{n,p,\mu,\nu}\left(  x;y;\dfrac{\eta}{t^{p}%
}\right)  t^{n} & =\alpha^{\varphi_{1}\left(  x\right)  }\psi \left(
x+t\right)  \beta^{-\varphi_{2}\left(  x+t\right)  }\\
& \times \Lambda_{\mu,\nu}(y;\eta)
\end{array}
\label{h2}%
\end{equation}
provided that each member of (\ref{h2}) exists.
\end{corollary}

\begin{remark}
\label{remark3.2}Getting $a_{k}=\frac{1}{k!},$ $\mu=0,$ $\nu=1$ and then
taking into account the generating function (\ref{aa*}) give the bilinear
generating function for the analytic functions $\Theta_{n}(x)$
\begin{align*}
&  \sum \limits_{n=0}^{\infty}%
%TCIMACRO{\dsum \limits_{k=0}^{\left[  n/p\right]  }}%
%BeginExpansion
{\displaystyle \sum \limits_{k=0}^{\left[  n/p\right]  }}
%EndExpansion
\frac{\Theta_{k}(y)}{k!}\frac{\Theta_{n-pk}(x)}{\left(  n-pk\right)  !}%
\eta^{k}t^{n-pk}\\
&  =\alpha^{\varphi_{1}\left(  x\right)  +\varphi_{1}\left(  y\right)  }%
\psi \left(  x+t\right)  \psi \left(  y+\eta \right)  \beta^{-\varphi_{2}\left(
x+t\right)  -\varphi_{2}\left(  y+\eta \right)  }%
\end{align*}

\end{remark}

Besides, it is possible to get multivariable function $\Omega_{\mu+\psi
k}(y_{1},...,y_{s}),$ $(s\in \mathbb{N})$ as an appropriate product of several
simpler functions for every suitable choice of the coefficients $a_{k}\,
\,(k\in \mathbb{N}_{0}).$ Thus, Theorem \ref{theorem3.1} can be applied in
order to derive various families of multilinear and multilateral generating
functions for the functions generated by (\ref{aa*}).

\end{document}